\documentclass[12pt]{amsart} 

\usepackage[OT2, T1]{fontenc}

\usepackage{amscd}
\usepackage{amsmath}
\usepackage{amssymb}
\usepackage{mathrsfs} 			% for \mathscr (script letters)
\usepackage{units}
\usepackage[all]{xy}

\usepackage{algorithm}
\usepackage{algorithmic}

\DeclareMathOperator{\ch}{char}
\DeclareMathOperator{\card}{card}

\def\frk{\frak}               % font for "Fraktur"

\def\mm{{\frk m}}

\def\Phi{{\frk n}}
\def\Phi{{\frk N}}
\def\opn#1#2{\def#1{\operatorname{#2}}} % to make operators
\opn\projdim{proj\,dim} \opn\injdim{inj\,dim} \opn\rank{rank}
\opn\depth{depth} \opn\sdepth{sdepth} \opn\fdepth{fdepth}
\opn\grade{grade} \opn\height{height} \opn\embdim{emb\,dim}
\opn\codim{codim}  \opn\min{min} \opn\max{max}

\opn\Tr{Tr} \opn\bigrank{big\,rank}
\opn\superheight{superheight}\opn\lcm{lcm}
\opn\trdeg{tr\,deg}%\emph{
\opn\reg{reg} \opn\lreg{lreg} \opn\ini{in} \opn\lpd{lpd}
\opn\size{size}
\opn\div{div} \opn\Div{Div} \opn\cl{cl} \opn\Cl{Cl}
\opn\Spec{Spec} \opn\Supp{Supp} \opn\supp{supp} \opn\Sing{Sing}
\opn\Ass{Ass} \opn\Min{Min}
\opn\Ann{Ann} \opn\Rad{Rad} \opn\Soc{Soc}
\opn\Im{Im} \opn\Ker{Ker} \opn\Coker{Coker} \opn\Am{Am}
\opn\Hom{Hom} \opn\Tor{Tor} \opn\Ext{Ext} \opn\End{End}
\opn\Aut{Aut} \opn\id{id}  \opn\deg{deg}

\opn\nat{nat}
\opn\pff{pf}%   \pf exists already
\opn\Pf{Pf} \opn\GL{GL} \opn\SL{SL} \opn\mod{mod} \opn\ord{ord}
\opn\Gin{Gin} \opn\Hilb{Hilb}
\opn\aff{aff} \opn\con{conv} \opn\relint{relint} \opn\st{st}
\opn\lk{lk} \opn\cn{cn} \opn\core{core} \opn\vol{vol}
\opn\link{link} \opn\star{star}
\opn\gr{gr}

\def\pot#1#2{#1[\kern-0.28ex[#2]\kern-0.28ex]}

\opn\dirlim{\underrightarrow{\lim}}
\opn\inivlim{\underleftarrow{\lim}}
\let\iso=\cong

\let\to=\rightarrow

\def\Implies{\ifmmode\Longrightarrow \else
        \unskip${}\Longrightarrow{}$\ignorespaces\fi}
\def\implies{\ifmmode\Rightarrow \else
        \unskip${}\Rightarrow{}$\ignorespaces\fi}
\def\iff{\ifmmode\Longleftrightarrow \else
        \unskip${}\Longleftrightarrow{}$\ignorespaces\fi}

\let\:=\colon
\newtheorem{Theorem}{Theorem}[]
\newtheorem{Lemma}[Theorem]{Lemma}
\newtheorem{Corollary}[Theorem]{Corollary}
\newtheorem{Proposition}[Theorem]{Proposition}

\theoremstyle{definition}

\newtheorem{Remark}[Theorem]{Remark}

\newtheoremstyle{subsection-tweak}
   {11pt}
   {3pt}%
   {}
   {}%
   {\bfseries}
   {}%
   {.5em}
   {\thmnumber{\@{#1}{}\@{#2}.}%
    \thmnote{~{\bfseries#3.}}}    

\newcounter{numberingbase}

\theoremstyle{subsection-tweak}
\newtheorem{bpp}[Theorem]{}
\newtheorem{bppt}[numberingbase]{}
\newcommand{\bbpp}{\begin{bpp}}
\newcommand{\eepp}{\end{bpp}}
\newcommand{\bbppt}{\begin{bppt}}
\newcommand{\eeppt}{\end{bppt}}

\theoremstyle{theorem}

\theoremstyle{definition}

\newcommand{\val}{\mathrm{val}}		% Valuation

\providecommand{\qxq}[1]{\quad\text{#1}\quad}

\newcommand{\tst}{\textstyle}

\newcommand{\sU}{\mathscr{U}}

\newcommand{\wt}{\widetilde}

\let\epsilon\varepsilon
\let\phi=\varphi
\def\qed{\ifhmode\textqed\fi
      \ifmmode\ifinner\quad\qedsymbol\else\dispqed\fi\fi}
\def\textqed{\unskip\nobreak\penalty50
       \hskip2em\hbox{}\nobreak\hfil\qedsymbol
       \parfillskip=0pt \finalhyphendemerits=0}
\def\dispqed{\rlap{\qquad\qedsymbol}}

\opn\dis{dis}
\def\pnt{{\raise0.5mm\hbox{\large\bf.}}}

\opn\Lex{Lex}

\begin{document}

\title{Valuation rings as limits of complete intersection rings}

\author{ Dorin Popescu}

\dedicatory{In the memory of J\"urgen Herzog}

\address{Simion Stoilow Institute of Mathematics of the Romanian Academy,
Research unit 5, P.O. Box 1-764, Bucharest 014700, Romania,}

\address{University of Bucharest, Faculty of Mathematics and Computer Science
Str. Academiei 14, Bucharest 1, RO-010014, Romania,}

\address{ Email: {\sf dorin.m.popescu@gmail.com}}

\begin{abstract} We show that a valuation ring  containing its residue field of characteristic $p>0$ is a filtered colimit of   complete intersection ${\bf F}_p$-algebras. 

 {\it Key words }: immediate extensions, pseudo convergent sequences, pseudo limits, ultrapowers, smooth morphisms, Henselian Rings.   \\
 {\it 2020 Mathematics Subject Classification: Primary 13F30, Secondary 13A18,13L05,13B40.}
\end{abstract}

\maketitle

\section*{Introduction}

The Zariski Uniformization Theory \cite{Z}  plays an important role in Desingularization Theorem.

\begin{Theorem}(Zariski) \label{T} Let $V$ be a  valuation ring containing a field $L\supset {\bf Q}$. Then $V$ is a filtered union of its regular local subrings essentially of finite type over $L$.
\end{Theorem}

This theorem says in fact that $V$ is  a filtered union of its smooth $L$-subalgebras  of finite type and
 the following theorem  is connected with it.

\begin{Theorem} \cite[Theorem 21]{P}\label{T0} If $V\subset V'$ is an immediate extension of valuation rings   containing $\bf Q$
then $V'$ is a filtered colimit of smooth $V$-algebras. 
\end{Theorem}

An inclusion $V \subset V'$ of valuation rings is an \emph{immediate extension} if it is local as a map of local rings and induces isomorphisms between the value groups and the residue fields of $V$ and $V'$. 
We recall that a  filtered colimit  is a limit indexed by a small category that is filtered
  (see \cite[002V]{SP} or \cite[04AX]{SP}). A filtered  union is a filtered direct limit in which all objects are subobjects of the final colimit,
   so that in particular all the transition arrows are monomorphisms.

   The above theorem does not hold when the valuation ring extension is algebraic and contains a field of positive characteristic (see \cite[Example 3.13]{Po1} inspired from \cite{O} and \cite[Remark 6.10]{Po1}), but  we got  the following result.

\begin{Theorem}\cite[Theorem 4]{P3}\label{T1} Let  $V\subset V'$ be an immediate  extension of valuation rings   containing  a field and $K\subset K'$ its fraction field extension. If $K'=K(x)$ for some algebraically independent system of elements $x$ over $K$ 
then $V'$ is a filtered union of smooth $V$-subalgebras of $V'$. 
\end{Theorem}

As in \cite[Proposition 18]{P1}(see also \cite[Proposition 1]{P4}) we have the following corollary.

\begin{Corollary}\cite[Corollary 24]{P3}\label{C0} Let $V\subset V'$ be an immediate extension of valuation rings containing a field and $K\subset K'$ its fraction field extension. If $K'=K(x)$ for some algebraically independent system of elements $x$ over $K$ 
and $V$ is Henselian then every finite system of polynomials over $V$, which has a solution in $V' $ has also one in $V$.
\end{Corollary}

The ideas of this corollary were used to show some conjectures of M. Artin in 
\cite[Theorems 1.3, 1.4]{Po0} (see also \cite[Theorem 5.3.1]{I}) using the General N\'eron desingularization (see  \cite[Theorem 2.5]{Po0}, \cite[Theorem 1.1]{S} and \cite[Theorem 5.2.56]{I}).

 A consequence of Theorem \ref{T1}   is a form of   Zariski's Uniformization Theorem in the special case of
 positive characteristic.
\begin{Corollary}(\cite[Corollary 5]{P3}, \cite[Theorem 1.3]{P5})\label{C} Let $V$ be a  valuation ring containing its residue field $k$   with a  value group $\Gamma$ free as a $\bf Z$-module, $\val$ its valuation and $K$ its fraction field. Assume that  $K=k(x,y)$ for some  algebraically independent elements $x,y$,  $x=(x_i)_{i\in I},y=(y_j)_{j\in J}$ over $k$ such that $\val(y)$ is a basis in $\Gamma$. 
 Then $V$ is a filtered  union of its smooth $k$-subalgebras.
\end{Corollary}
 This Corollary says  that $V$ is  a filtered union of its regular local subrings essentially of finite type over $k$.

In general, we got the following theorem based on Theorem \ref{T1} and \cite[Theorem 1]{P2}.
 
\begin{Theorem} (\cite[Theorem 6]{P3}, \cite[Theorem 4.3]{P5})\label{T2} Let $V\subset V'$ be an  immediate extension of valuation rings containing a field.   Then $V'$ is a filtered
 union of its complete intersection $V$-subalgebras of finite type.
\end{Theorem}

A {\em complete intersection} $V$-algebra of finite type is a  $V$-algebra of type $C/(P)$, where $C$ is  a polynomial $V$-algebra and $P$ is a regular  sequence in  $C$.  Thus the above theorem says that $V'$ is a  filtered
 union of its $V$-subalgebras of type $C/(P)$. Since $V'$ is local it is enough to say that $V'$ is  a  filtered
 union of its $V$-subalgebras of type $T_h/(P)$,  $T$ being a polynomial $V$-algebra, $0\not =h\in T$ and $P$  a regular sequence in $T$. Clearly, $ T_h$ is a smooth $V$-algebra and in fact it is enough to say that $V'$ is  a  filtered
 union of its $V$-subalgebras of type $G/(P)$, where $G$  is a smooth $V$-algebra and $P$ is a regular system of elements of $G$. Conversely, a $V$-algebra of such type $G/(P)$ has the form $T_h/(P)$ for some $T,h,P$ using \cite[Theorem 2.5]{S}. By abuse we understand by a {\em complete intersection } $V$-algebra a $V$-algebra of such type $G/(P)$, or $T_h/(P)$ which are not assumed to be flat over $V$.

Like Corollary \ref{C} we have the following corollary. 

\begin{Corollary}(\cite[Corollary 7]{P3}, \cite[Corollary 4.4]{P5}, see also \cite[Theorem 19]{P4})\label{C1} Let $V$ be a  valuation ring containing its residue field $k$   with a free value group  $\Gamma$.  
 Then $V$ is a filtered  union of its complete intersection $k$-subalgebras of finite type.
\end{Corollary}

\begin{Remark} \label{r0} A finitely generated ordered group is free.
\end{Remark}

The goal
of this paper is to extend   Corollary \ref{C} (see Theorems \ref{t3}, \ref{t4}, \ref{t-1}) and
 to establish the following theorem, that is   an extension of the above corollary for the case when $\Gamma$ is not finitely generated.

\begin{Theorem}\label{T3} Let $V$ be a  valuation ring containing a perfect field $F$ of positive characteristic, $k$ its residue field and   $\Gamma$ its  value group.  Then $V$ is a filtered
colimit of  complete intersection $F$-algebras of finite type if one of the following conditions holds:
\begin{enumerate}

\item  $k\subset V$, 
\item  $V$ is Henselian. 
\end{enumerate}
\end{Theorem}

The proof is given in  Theorems \ref{t1}, \ref{t2}. It is possible that the above theorem could give a complete intersection version of 
the Desingularization Theorem.

  When $\Gamma$ is not finitely generated we have to use  the cross-sections (these are retractions $\Gamma \to K^*$ of the valuation) as in    \cite[Lemma 7.9]{v}, or  \cite[3.3.39, 3.3.40]{ADH} obtained using  ultrapower methods as we did it in \cite{P} when $V\supset {\bf Q}$. The same methods are necessary to reduce to the case when the residue field $k\subset V$ in the case when $V$ is Henselian (see Theorem \ref{bt1}). 

The above theorem connects the theory of valuation rings with the theory of complete intersection local rings. We remind that for a given complete local ring $(R,\mm,k)$ a minimal free resolution over $R$ of the residue field $k$ is described in \cite[Theorem 4]{Ta} as well as in
 \cite[Theorem 2.7]{As}, \cite[§6]{Av}, \cite[Proposition 1.5.4]{GL},  \cite[§3]{HM}, \cite[ Theorem 2.5]{AHS}, \cite[Theorem 4.1]{NV}.

We owe thanks to some anonymous referee who had the useful comments on our paper.

\vskip 0.5 cm

\section{Filtered colimits of smooth algebras}

A possible extension of Zariski's Uniformization Theorem is the following theorem.

\begin{Theorem} \label{TAD}
 (B. Antieau, R. Datta \cite[Theorem 4.1.1]{AD}) Every perfect valuation ring of characteristic $p>0$ is a filtered union of its smooth ${\bf F}_p$-subalgebras. 
 \end{Theorem}

Using the above theorem we get a consequence of Theorem \ref{T1} similar to Corollaries \ref{C}, \ref{C1}.

\begin{Corollary}  Let $V\subset V'$ be an immediate extension of valuation rings of positive characteristic $p$
and $K\subset K'$ their fraction field extension. Assume that $K=K^p$ and $K'$ is a pure transcendental extension of $K$. Then $V'$ is a  filtered union of its smooth ${\bf F}_p$-subalgebras. If $K'/K$ is not a pure transcendental extension
then $V'$ is  a filtered union of its complete intersection  ${\bf F}_p$-subalgebras of finite type.
 \end{Corollary}

For the proof of the first statement, note that $V$ is a filtered union of its smooth ${\bf F}_p$-subalgebras by Theorem \ref{TAD} and $V'$ is a filtered union of its smooth $V$-subalgebras by Theorem \ref{T1}. The proof of the second statement goes similarly using Theorem \ref{T2} instead Theorem \ref{T1}. It is also necessary to note that a complete intersection algebra over a complete intersection algebra is still a complete intersection algebra by  \cite[Lemma 6]{P2}.

We recall below    \cite[Lemma 7]{P}, which is an extension of   
 \cite[Proposition 3]{KPP}, and \cite[Proposition 5]{Po2}.

\begin{Lemma} \label{k'}
For a commutative diagram of ring morphisms

\xymatrix@R=0pt{
& B \ar[rd] & & && B \ar[dd]^{b \, \mapsto\, a} \ar[rd] & \\
A \ar[rd] \ar[ru] & & V & \mbox{that factors as follows} & A \ar[ru]\ar[rd] & & V/a^3V \\ 
& A' \ar[ru] & & && A'/a^3A' \ar[ru] &
}
\noindent with $B$ finitely presented over $A$, a $b \in B$ that is standard over $A$ (this means a special element from the ideal $H_{B/A}$ defining the non smooth locus of $B$ over $A$, for details see for example \cite[Lemma 4]{P}), and a nonzerodivisor $a \in A'$ that maps to a nonzerodivisor in $V$ that lies in every maximal ideal of $V$,
there is a smooth $A'$-algebra $S$ such that the original
diagram factors as follows:

\hskip 4 cm\xymatrix@R=0pt{
& B \ar[rdd] \ar[rrd] & & \\
A \ar[rd] \ar[ru] & &  & V. \\
& A' \ar[r]  & S \ar[ru] &
}

 \end{Lemma}

 Let $V$ be a valuation ring,  $\lambda$  a fixed limit ordinal  and $v=\{v_i \}_{i < \lambda}$ a sequence of elements in $V$ indexed by the ordinals $i$ less than  $\lambda$. Then $v$ is \emph{pseudo convergent} if $\val(v_{i} - v_{i''} ) < \val(v_{i'} - v_{i''} )     \ \ \mbox{for} \ \ i < i' < i'' < \lambda$
(see \cite{Kap}).
A  \emph{pseudo limit} of $v$  is an element $v \in V$ with $ \val(v - v_{i}) < \val(v - v_{i'}) \ \ \mbox{(that is,} \ \ \val(v -  v_{i}) = \val(v_{i} - v_{i'})) \ \ \mbox{for} \ \ i < i' < \lambda$.

When for any $\gamma\in \Gamma$ it holds that $\val(v_{i} - v_{i'})>\gamma$ for  $ i < i' < \lambda$
large we call $v$ {\em fundamental}. $V$ is {\em complete} if every fundamental sequence over $V$ has a pseudo limit in $V$ (in this case it is unique and  is called {\em limit}). The completion 
of a valuation ring $V$ is the ring formed  by all the fundamental
sequences with the obvious operations of addition and multiplication, modulo the natural equivalence relation.
 An immediate extension $V\subset V'$ of valuation rings is dense if any element of $V'$ is a limit of a fundamental sequence of $V$. Thus if $V\subset V'$ is dense then $V,V'$  have  the same completion.

Next we recall  an application of Lemma \ref{k'}. 

\begin{Proposition}(\cite[Proposition  9]{P} \label{dense}
For a  ring $A$, a dense extension of valuation rings (see Section 3) $V\subset V'$, $K$ the  fraction field of $V$,  a ring morphism $A \to V$, a finitely presented $A$-algebra $B$, and maps
\[
A \to B \to V \ \ \mbox{such that} \ \ B \to K \ \ \mbox{factors through some $A$-smooth localization of $B$}
\] 
suppose that
 there exist a smooth $A$-algebra $S'$ and a factorization $
A \to B \to S' \to V'$. Then
 there exist a smooth $A$-algebra $S$ and a factorization $
A \to B \to S \to V$. 
In particular, there exist a smooth $A$-algebra $S$ and a factorization $A \to B \to S \to V$   if  there exist a smooth $A$-algebra $\hat S$ and a factorization $A \to B \to {\hat S} \to {\hat V}$,  ${\hat V}$ being the completion of $V$. 
\end{Proposition}

\begin{Remark} \label{R} In the spirit  of the above proposition, let $V\subset V'$ be a separable dense extension of valuation rings, $K\subset K'$ its fraction field extension, $B$ a finitely presented $V$-algebra and $\rho:B\to V'$ a $V$-morphism.
We may change $B$ by the image of $\rho$. It remains finitely presented over $V$ by \cite[Theorem 4]{Na}.
 As $K'/K$ is separable there exists a smooth $V$-algebra $S''$ such that the composite map 
$B\xrightarrow{\rho} V'\to K'$ factors through $S''$. Thus $\rho(H_{B/V})V'\not =0$ and so   $\rho(H_{B/V})\not =0$ 
and $V/\rho(H_{B/V})V\cong V'/\rho(H_{B/V})V'$, since $V\subset V'$ is dense. Choose a standard element $d\in V$ for $B$ over $A$.
Then $\rho$ factors through a smooth $V$-algebra $S'$  using Lemma \ref{k'} applied to $A=A'=V$ and $d,V'$ because $V/d^3V\iso V'/d^3V'$. 
\end{Remark}
 
The following theorem follows from  the above remark, Theorem \ref{T1} and the proof of Corollary \ref{C}. 

\begin{Theorem} \label{t3} Let $V\subset V'$ be an immediate extension of valuation rings of positive characteristic, $\val$ the valuation of $V'$ 
and $K\subset K'$ their fraction field extension. Assume that 

\begin{enumerate}

\item $V$ contains its residue field $k$ and $x$ is a system of elements of $V$ such that $\val(x)$ forms a $\bf Z$-linear basis of the value group of $V$.

\item $V$ is a separable dense extension of $V\cap k(x)$.

\item $K'$ is a separable dense extension of a pure transcendence field extension of $K$.
\end{enumerate}

Then $V'$ is a filtered colimit of smooth $k$-algebras.

\end{Theorem}

\begin{proof} As in the proof of Corollary \ref{C}, $V\cap k(x)$ is a filtered union of some localizations of polynomial algebras over $k$. 
 Let $y$ be a transcendental basis of $K'$ over $K$ such that $K(y)\subset K'$ is dense. The separable dense extensions $V\cap k(x)\subset V$ and $V'\cap K(y)\subset V'$ are filtered colimits of smooth $V\cap k(x)$-algebras, respectively $V'\cap K(y)$-algebras by the above remark. It remains to apply Theorem \ref{T1} to the extension $V\subset V'\cap K(y)$.
 \hfill \end{proof}

An extension of integral domains is {\em separable} if  its fraction field extension is separable.

Using Theorem \ref{TAD} we get a similar theorem.

\begin{Theorem} \label{t4} Let $V\subset V'$ be an immediate extension of valuation rings of positive characteristic $p$, $\val$ the valuation of $V'$ 
and $K\subset K'$ their fraction field extension. Assume that 

\begin{enumerate}

\item There exists a valuation subring $W$ of $V$ such that $W=W^p$.
\item $V$ is a separable dense extension of $W$.

\item $K'$ is a separable dense extension of a pure transcendental field extension of $K$.
\end{enumerate}

Then $V'$ is a filtered colimit of smooth ${\bf F}_p$-algebras.

\end{Theorem}

In the case of one dimensional valuation extension there exists a stronger form of Theorem \ref{t3}, which we will present next together with some necessary facts concerning the    defectless of a certain valuation ring.

Let $V$ be a  valuation ring with value group $\Gamma$,  $K$ its fraction field and  its valuation $\val$. Let $K'$ be a finite field extension of $K$, $v_i$, $1\leq i\leq r$ the valuations of $K'$ extending $\val$ and $V_i$  the valuation rings of $K'$ defined by $v_i$. Let $e_i, f_i$ be the ramification index, respectively the degree of the residue field extension of $V_i$ over $V$. It is well known that 
$$[K':K]\geq \sum_{i=1}^r e_if_i.$$
 If this inequality is an equality for all finite  field  extensions $K'$ of $K$ we say that $V$ is {\em defectless} (see \cite{K} for  details and examples).

 Very useful is the following  particular form of the Generalized Stability Theorem of F. V. Kuhlmann (see \cite[Theorem 1.1]{K1}, \cite[Theorem 5.1]{K}).

\begin{Theorem}(Kuhlmann)\label{k}  Let $V\subset V'$ be an extension of   valuation rings   with the same residue field, $\Gamma\subset \Gamma'$ its value group extension, $K\subset K'$ its fraction field and $\val$ its valuation.  Assume  $K'/K$ is a finite type field extension and  $\Gamma'/\Gamma$ is a finitely generated   free $\bf Z$-module  of rank   $\trdeg K'/K $. If $V$ is defectless  then $V'$ is  defectless,  too. 
\end{Theorem}

\begin{Corollary}\label{ku}  Let $V$ be a  valuation ring of positive characteristic   with  finitely generated value group $\Gamma$, $\val$ its valuation and $K$ its fraction field. Assume  $V$ contains its residue field $k$ and $K=k(y)$ for some elements $y=(y_1,\ldots,y_r)$ of $V$ such that $\val(y_i)$, $1\leq i\leq r$ is a basis of the free $\bf Z$-module $\Gamma$.
Then $V$ is  defectless. 
\end{Corollary}
For the proof apply the above theorem to $k\subset V$ since $k$ 
is defectless for the trivial valuation.

The following lemma is \cite[Lemma 8]{P-1} but we  sketch a proof here.

\begin{Lemma}\label{1}
 Let $V\subset V'$ be an immediate extension of one dimensional valuation rings and $K\subset K'$ their fraction field extension. If $V$ is defectless and $K'/K$ is  algebraic then $ K'/K$ is dense.
 \end{Lemma}
 \begin{proof} We may consider only the case when  $K'/K$ is finite. Let $K^h$, $K'^h$ be the Henselizations of $K$ respectively $K'$ and $L=K^h(K')\subset K'^h$.  Note that $K^h$ is  defectless because $K$ is so (see \cite[Theorem 2.3]{K}).
 Then $[L:K^h]=e_{L/K^h}f_{L/K^h}=1$ because there exists a unique extension of the valuation of  $K^h$  to $L$ and so $K'\subset L\subset K^h$.
 But  $K\subset K^h$ is dense because $\dim V=1$. 
\hfill\ \end{proof}

\begin{Theorem} \label{t-1} Let $V\subset V'$ be an immediate  extension of one dimensional valuation rings of positive characteristic, $\val$ the valuation of $V'$ 
and $K\subset K'$ their fraction field extension. Assume that 

\begin{enumerate}

\item $V$ contains its residue field $k$ and $x$ is a finite system of elements of $V$ such that $\val(x)$ forms a $\bf Z$-linear basis of the value group of $V$.

\item There exists an algebraic separable field extension $L\subset K$ of $k(x)$ such that 
 $V$ is a separable dense extension of $V\cap L$.

\item $K'$ is a separable dense extension of a pure transcendental field extension of $K$.
\end{enumerate}

Then $V'$ is a filtered colimit of smooth $k$-algebras.

\end{Theorem}

\begin{proof}
The proof follows as in Theorem \ref{t3}. We should see  that the extension $L/k(x)$ is dense. This follows from Lemma \ref{1} and Corollary \ref{ku}.
\hfill\ \end{proof}

\begin{Remark} If the value group of $V$ is not finitely generated then the above theorem is false, because the algebraic separable extension $L/k(x)$ could be  not dense, as shows \cite[Example 9]{P-1}, or \cite[Example 8]{P4}.
\end{Remark}

\vskip 0.5 cm

\section{Model Theory of valued fields}

The goal of the next corollary is   to extend Corollary \ref{C1}.

\begin{Corollary}\label{c1} 
Let $V$ be a  valuation ring containing a field $F$ of positive characteristic, $k$ its residue field and   $\Gamma$ its  value group.   Assume that  $\Gamma$ is finitely generated,  $k/F$ is separably generated  (that is there exists a transcendental basis $x$ of $k$ over $F$ such that the field extension $F(x)\subset k$ is algebraic separable) and  $V$ is Henselian.    Then  $V$ is a filtered
 union of its complete intersection $F$-subalgebras.
\end{Corollary}
\begin{proof} The conclusion follows by Corollary \ref{C1} if we show that $k\subset V$.  A lifting of $k$ to $V$ can be done when $V$ is Henselian and $k/F$ is separably generated. 
\hfill\ \end{proof}

\begin{Remark} \label{r0'} The above lifting follows by \cite[Theorem 2.9]{v} when $\ch(k)=0$. The proof is the same in positive characteristic when $k/F$ is separably generated.
\end{Remark}

\begin{Corollary}\label{cf} 
Let $V$ be a  valuation ring containing a  field $F$ of positive characteristic, $k$ its residue field,   $\Gamma$ its  value group and $K$ its fraction field.   Assume that there exists a cross-section  $s:\Gamma\to K^*$   and either $k\subset V$, or   $k/F$ is separably generated  and  $V$ is Henselian.    Then  $V$ is a filtered
 union of its complete intersection $F$-subalgebras.
\end{Corollary}
 The proof goes as in Corollaries \ref{C1}, \ref{c1} with $V_0=V\cap k(s(\Gamma))$. More precisely, $V_0$ is a filtered union of $V\cap k(s(\Gamma'))$  for all finitely generated subgroups (so free)  $\Gamma'$ of $\Gamma$, which are localizations of polynomial $k$-algebras.

There are two problems to reduce to Corollaries \ref{C}, \ref{c1}, \ref{cf}. These are to reduce to the case when the following conditions hold:

i)   $V$ has a cross-section $s:\Gamma\to K^*$, and

ii)  $V$ contains its residue field.

For i) we may use  \cite[Proposition 5.4, Lemma 7.9]{v}, or \cite[3.3.39, 3.3.40]{ADH}. 
In \cite[Theorem A.10]{P} there exists the following  variant of these results.

\begin{Theorem} \label{bt}
For a valuation ring $V$ with value group $\Gamma$, there is a countable sequence of ultrafilters ${\sU}_1, {\sU}_2, \dots$ on some respective sets $U_1, U_2, \dots$ for which the valuation rings $\{V_n\}_{n \geq 0}$ defined inductively by $V_0 := V$ and $V_{n + 1} := \prod_{{\sU}_{n + 1}} V_{n}$ %for $n \geq 1$ 
are such that the valuation ring
\[
\tst \wt{V} := \varinjlim_{n \geq 0} V_n \qxq{admits a cross-section}  \wt{s} : \wt{\Gamma} \to \wt{K}^*,
\]
where $\wt{K}$ and $\wt{\Gamma}$ are the fraction field and the value group  of $\wt{V}$, respectively.
\end{Theorem}

\begin{Theorem}\label{t1}
Let $V$ be a  valuation ring  containing its residue field $k$ of positive characteristic and   $\Gamma$ its  value group.   Then  $V$ is a filtered
colimit of  complete intersection $k$-algebras.  Moreover,
 $V$ is a filtered  colimit of   complete intersection  ${\bf F}_p$-algebras  of finite type.
\end{Theorem}

\begin{proof} 
 We may assume that $V$ is not a field, otherwise the conclusion is trivial. Let $E$ be a finitely generated $ k$-subalgebra (in fact finitely presented by \cite[Theorem 4]{Na}) of $V$ and $w:E\to V$  a morphism. We will show that $w$ factors through a complete intersection  $V$-algebra. It is enough to replace $V$ by the intersection of $V$ with the fraction field $K_E$ of $E$.
 
  As $K_E$ is a finite type field extension of $k$  we may assume   $\Gamma$ is countably generated  by \cite[Corollary 5.2]{APZ}. Consider  $U_i$, ${\sU}_i$, $\wt{V}$, $\wt{k}$, $\wt{\Gamma}$, $\wt{s}$ given by Theorem \ref{bt}  for  $V$. Note that $\wt{k}\subset \wt{V}$ because $k\subset V$. As usual $\wt{V}$ is an immediate extension of the valuation ring $\wt{V}_0=\wt{V}\cap \wt{k}(\wt{s}(\wt{\Gamma}))$,  which is a filtered  union of localizations of polynomial algebras over $\wt{k}$. Using Theorem \ref{T2}  we see that $\wt{V}$ is a filtered  union of its complete intersection $\wt{V}_0$-subalgebras. 
  
  Then the map $w' :\wt{k}\otimes_F E\to \wt{V}$ factors through a complete intersection $\wt{k}$-algebra and so the map $E\to \wt{V}$ induced by $w$ factors through a complete intersection $k$-algebra $D$ since $\wt{k}/k$ is separable.
   Thus $w$ factors through $D$ too   because all finite systems of polynomial equations which have a solution in $\tilde V$ must have one in $V$. 
 The proof ends applying Remark \ref{s} below.
\hfill\ \end{proof}

\begin{Remark} \label{s} (\cite[Lemma 1.5]{S})  Let $A\subset B$ be an extension of rings. Then $B$ is a filtered colimit of  complete intersection $A$-algebras if and only if for each finitely generated $A$-algebra $C$ and each $A$-morphism $w:C\to B$ there exists a complete intersection $A$-algebra $D$ such that $w$ factors through $D$.
\end{Remark}

The problem from ii) is also hard. As we explain in the proof of Corollary \ref{c1} the lifting of the residue field $k$ of $V$ when the characteristic $p>0$ could be done when $V$ is Henselian and $k/{\bf F}_p$ is of finite type. To glue the liftings of finite type subfields of $k$ is another problem which is solved using Model Theory.

We need \cite[Proposition A.6]{P} (useful also in the proof of Theorem \ref{bt}), which is obtained using \cite[Theorem 6.1.4]{CK} and says in particular the following:

\begin{Proposition}\label{kes} Let $V$ be a valuation ring with value group $\Gamma$. Then there exists an ultrafilter ${\sU}$ on a set $U$ such that any system of polynomial equations
$(g_i((X_j)_{j\in J})_{i\in I}$ with $\card(I)\leq  \card(U)$ in variables $(X_j)_{j\in J}$ with coefficients in  the ultrapower ${\tilde V}=\Pi_{\sU}V$ has a solution in  ${\tilde V}$ if and only if all its 
finite subsystems do.
\end{Proposition}

\begin{Theorem} \label{bt1}
For a Henselian valuation ring $V$ with residue field $k$ of characteristic $p>0$, there is a countable sequence of ultrafilters ${\sU}_1, {\sU}_2, \dots$ on some respective sets $U_1, U_2, \dots$ for which the valuation rings $\{V_n\}_{n \geq 0}$ defined inductively by $V_0 := V$ and $V_{n + 1} := \prod_{{\sU}_{n + 1}} V_{n}$ for $n \geq 1$ 
are such that for the valuation ring 
$ \wt{V} := \varinjlim_{n \geq 0} V_n$ there exists a lifting  $  \wt{t} : \wt{k} \to \wt{V}$,
where  $\wt{k}$ is the residue field  of $\wt{V}$.
\end{Theorem}

\begin{proof}
Let  $\mathcal{E}$ be the family of all subfields  $k'$ of $k$ which are field extensions of finite type over ${\bf F}_p$. For such a subfield $k'$ there exists a lifting $t_{k'}:k' \to V$ as we have seen above since $k'/{\bf F}_p$ is separably generated. We have $k'= {\bf F}_p({\bar y}_{k'})$ for a finite system of elements ${\bar y}_{k'}$ from $k$. Let $G_{k'}$ be the polynomial equations from ${\bf F}_p[Y]$ satisfied by  
${\bar y}_{k'}$ (certainly $G_{k'}$ could contain just one equation if we arrange properly ${\bar y}_{k'}$ but it has no importance for us). The existence of $t_{k'}$ says that $G_{k'}$ has a solution $y_{k'}$   in $V$.
If $k'_1\subset k'_2$ for some fields $k'_i\in \mathcal{E}$, $i=1,2$ then the elements ${\bar y}_{k'_1},{\bar y}_{k'_2}$ satisfy some polynomial equations $G_{k'_1k'_2}$ from ${\bf F}_p[Y]$,  which include $G_{k'_i}$ for $i=1,2$. The existence of $t_{k'_2}$ assure us that $G_{k'_1k'_2}$ has a solution in $V$. Let $G$ be the union of all possible  $G_{k'_1}$, $G_{k'_1k'_2}$ when $k'_1\subset k'_2$ run over $\mathcal{E}$.

Let  $U_1$ be a set with card $U_1>$card $k$ and ${\sU}_1$ an ultrafilter given by Proposition \ref{kes}. Then $G$ has a solution in  $V_1= \prod_{{\sU}_1} V$ because each finite subsystem of $G$  already  has a solution in $V$ 
given by the lifting of  a big enough $k'\in \mathcal{E}$. So we get a lifting $t_k$ of $k$ to $V_1$.

Repeating this procedure we  find  sets $(U_n)$ and  ultrafilters $({\sU}_n)$ and define  $V_n$ and $k_n$, so that $k_{n + 1} \cong \prod_{{\sU}_{n + 1}} k_{n}$ and $V_{n + 1} = \prod_{{\sU}_{n + 1}} V_{n }$  with
\[
\tst \wt{k} \cong \varinjlim_{n \geq 0} k_n \ \ \mbox{and} \ \ \wt{V} \cong \varinjlim_{n \geq 0} V_n.
\]
We  build the ultrafilters ${\sU}_n$ one by one  in such a way that the desired map
\[
\tst \wt{t} \colon \wt{k} \to \wt{V} \ \ \mbox{is the limit of compatible maps } \ \ t_n \colon k_n \to V_{n + 1}.
\]
\hfill\ \end{proof}

\begin{Remark} \label{r} The proofs of Theorems \ref{bt}, \ref{bt1} are similar and in fact we can find $\wt{V}$ such that there exist a cross-section $\wt{s}:\wt{\Gamma}\to \wt{K}^*$ and a lifting $\wt{t}:\wt{k}\to \wt{V}$ at the same time.
\end{Remark}

\begin{Theorem}\label{t2}
Let $V$ be a  valuation ring  of  characteristic $p>0$. 
Let  $F$ be a field contained in $V$ such the residue field $k$ of $V$ is separable over $F$ (for example ${\bf F}_p$).  
Assume that $V$ is Henselian.
 Then  $V$ is a filtered
colimit of  complete intersection $F$-algebras.
\end{Theorem}

\begin{proof}  Let   $\Gamma$ be the  value group of $V$.
  Let $E$ be a finitely generated $ F$-subalgebra (in fact finitely presented by \cite[Theorem 4]{Na})  of $V$ and $w:E\to V$ the inclusion.   We will show that $w$ factors through a complete intersection  $V$-algebra. It is enough to replace $V$ by the intersection of $V$ with the fraction field $K_E$ of $E$. Thus we may assume that $\Gamma$ is countably generated and $k$ is countably generated over $ F$ by \cite[Corollary 5.2]{APZ}. Consider  $U_i$, $\mathcal{U}_i$, $\wt{V}$, $\wt{k}$, $\wt{\Gamma}$, $\wt{s}$, $\wt{t}$ given by Theorems \ref{bt}, \ref{bt1}  for  $V$.  Thus we may assume $\wt{k}\subset \wt{V}$ (see Theorem \ref{bt1} and Remark \ref{r}) and so it is enough to apply Theorem \ref{t1}. As usual $\wt{V}$ is an immediate extension of the valuation ring
   $\wt{V}_0=\wt{V}\cap \wt{k}(\wt{s}(\wt{\Gamma}))$,  which is a filtered  union of localizations of polynomial algebras over $\wt{k}$.
  
    Then the map $w' :\wt{k}\otimes_F E\to \wt{V}$   factors through a complete intersection $\wt{k}$-algebra and so the map $E\to \wt{V}$ induced by $w$ factors through a complete intersection $F$-algebra $D$ since $\wt{k}/k$ and $k/F$  are separable.

   Thus $w$  factors through $D$ too    because all finite systems of polynomial equations which have a solution in $\tilde V$ must have one in $V$.
 The proof ends applying Remark \ref{s}.
\hfill\ \end{proof}

\end{document}